\documentclass[10pt]{amsart} 
\usepackage[all]{xy}     
\usepackage{amssymb}
\usepackage[english]{babel}
\newtheorem{defi}{\bf D\scriptsize EFINITION \normalsize}
\newtheorem{theorem}{\bf T\scriptsize HEOREM \normalsize}
\newtheorem{lm}{\bf L\scriptsize EMMA \normalsize}
\newtheorem{dk}{\bf C\scriptsize OROLLARY \normalsize}
\newtheorem{rem}{\bf R\scriptsize EMARK \normalsize}
\newtheorem{exa}{\bf E\scriptsize XAMPLE \normalsize}
\newtheorem{pro}{\bf P\scriptsize ROBLEM \normalsize}
\newtheorem{prop}{\bf P\scriptsize ROPOSITION \normalsize}
\newtheorem{no}{\bf N\scriptsize OTE \normalsize}

\newenvironment{remark}{\begin{rem}\rm}{\end{rem}}
\def\kopr{\hfill\raisebox{3pt}{\framebox{$\star$}}}

\begin{document}

{\hskip 115mm  December 7, 2010 
\vskip 5mm}

\title{Topological entropy and irregular recurrence}

\author{Lenka Obadalov\'a}

\address{ Mathematical Institute, Silesian University, 
CZ-746 01 Opava, Czech Republic}

\email{lenka.obadalova@math.slu.cz}

\thanks{The research was supported, in part, by grant SGS/15/2010 from  the Silesian University in Opava.}

\begin{abstract} 
This paper is devoted to problems stated by Z. Zhou and F. Li in 2009. They concern relations between almost periodic, weakly almost periodic, and quasi-weakly almost periodic points of a continuous map $f$ and its topological entropy. The negative answer follows by our recent paper. But for continuous maps of the interval and other more general one-dimensional spaces we give more  results; in some cases, the answer is positive.

\end{abstract}

\maketitle

%\pagestyle{myheadings}
%\markboth{{\sc Lenka Obadalov\'a}}
%{{\sc Irregular Recurrence}}

\section{Introduction}
\bigskip
 Let $(X,d)$ be a compact metric space, $I=[0,1]$ the unit interval, and $\mathcal C(X)$ the set of continuous maps $f :X\rightarrow X$.  By $\omega (f,x)$ we denote the {\it $\omega$-limit set} of $x$ which is the set of limit points of the {\it trajectory} $\{f ^{i}(x)\} _{i\ge 0}$ of $x$, where $f^{i}$ denotes the $i$th iterate of $f$. We consider sets $W(f)$ of {\it weakly almost periodic points} of  $f$, and $QW(f)$ of {\it quasi-weakly almost periodic points} of  $f$. They are defined as follows, see \cite{zhou}:
$$
W(f)= \left\{x \in X; \forall \varepsilon ~\exists N>0 \ \text{such that} \ \sum_{i=0}^{nN-1} \chi_{B(x,\varepsilon)}(f^i(x)) \geq n, \forall n >0\right\},
$$
$$
QW(f)= \left\{x \in X; \forall \varepsilon ~\exists N>0, \exists \{n_j\} \ \text{such that} \  \sum_{i=0}^{n_jN-1} \chi_{B(x,\varepsilon)}(f^i(x)) \geq n_j, \forall j >0\right\},
$$
where ${B(x,\varepsilon)}$ is the $\varepsilon$-neighbourhood of $x$,  $\chi_{A}$ the characteristic 
function of a set $A$, and $\{n_j\}$ an increasing 
sequence of positive integers. For $x \in X$ and $t>0$,  let
\begin{eqnarray}
\label{novew}
\Psi_x (f, t) &=& \liminf_{n \to \infty} \tfrac1n \# \{0\le j<n; d(x, f^j(x))<t\}, \\
\label{noveqw}
\Psi_x^* (f, t) &=& \limsup_{n \to \infty} \tfrac1n \# \{0\le j<n; d(x, f^j(x))<t\}.
\end{eqnarray}
Thus, $\Psi _x(f,t)$ and $\Psi ^*_x(f,t)$ are the {\it lower} and {\it upper Banach density} of the set $\{ n\in\mathbb N; f^n(x)\in B(x,t)\}$, respectively. In this paper we make of use more convenient definitions of $W(f)$ and $QW(f)$ based on the following lemma.

\begin{lm}
\label{l1}
Lef $f \in \mathcal C(X)$. Then

 {\rm (i)} \ $x \in W(f)$ if and only if $\Psi_x (f, t)>0$, for every $t>0$,

{\rm (ii)} \ $x \in QW(f)$ if and only if $\Psi_x^* (f, t)>0$, for every $t>0$.
\end{lm}

\begin{proof}
It is easy to see that, for every $\varepsilon>0$ and $N>0$,
\begin{equation}
\label{99}
\sum_{i=0}^{nN-1} \chi_{B(x,\varepsilon)}(f^i(x)) \geq n \
\ \ \text{if and only if} \ \ \ \#\{0 \leq j <nN; f^j(x) \in B(x, \varepsilon)\} \geq n.
\end{equation}
(i) If $x\in W(f)$ then, for every $\varepsilon >0$ there is an $N>0$ such that  the condition on the left side in (\ref{99}) is satisfied for every $n$. Hence, by the condition on the right, $\Psi _x(f,\varepsilon )\ge 1/N>0$. If $x\notin W(f)$
then there is an $\varepsilon >0$ such, that for every $N>0$, there is an $n>0$ such that the condition on the left side of (\ref{99}) is not satisfied. Hence, by the condition on the right,  $\Psi _x(f,t)<1/N\to 0$ if $N\to\infty$. Proof of (ii) is similar.
\end{proof}

\bigskip

Obviously, $W(f)\subseteq QW(f)$. The properties of $W(f)$ and $QW(f)$ were studied in the nineties by Z. Zhou et al, see \cite{zhou} for references. The points in $IR(f):=QW(f)\setminus W(f)$ are {\it irregularly recurrent points}, i.e., points $x$ such that $\Psi_x^*(f, t) >0 $ for any $t>0$, and $\Psi_x(f, t_0)=0$ for {\it some} $t_0>0$, see \cite{lenka}. 
Denote by $h(f)$ {\it topological entropy} of $f$ and by $R(f)$, $UR(f)$ and $AP(f)$ the set of {\it recurrent}, {\it uniformly recurrent} and {\it almost periodic} points of $f$, respectively. Thus, $x\in R(f)$ if, for every neighborhood $U$ of $x$, $f^j(x)\in U$ for infinitely many $j\in\mathbb N$, $x\in UR(f)$ if, for every neighborhood $U$ of $x$ there is a $K>0$  such that every interval $[n, n+K]$ contains a $j\in\mathbb N$ with $f^j(x)\in U$, and  $x\in AP(f)$, if for every neighborhood $U$ of $x$, there is a $k>0$ such that $f^{kj}(x)\in U$, for every $j\in\mathbb N$. Recall  that $x\in R(f)$ if and only if $x\in\omega (f,x)$, and $x\in UR(f)$ if and only if $\omega (f,x)$ is a {\it minimal set}, i.e.,  a closed set $\emptyset\ne M\subseteq X$ such that $f(M)=M$ and no proper subset of $M$ has this property. Denote by $\omega (f)$ the union of all $\omega$-limit sets of $f$.
The next relations follow by definition:
\begin{equation}
\label{eq10}
AP(f)\subseteq UR(f)\subseteq W(f)\subseteq QW(f)\subseteq R(f)\subseteq \omega (f)
\end{equation}
The next theorem will be used in Section 2. Its part (i) is proved in \cite{zhou2} but we are able to give a simpler argument, and extend it to part (ii).

\begin{theorem}
If $f\in\mathcal C(X)$ then 

{\rm (i)} \  $W(f) = W(f^m)$,

{\rm (ii)} \  $QW(f) = QW(f^m)$,

{\rm (iii)} \ $IR(f) = IR(f^m)$.
\label{dalsi3vlastnosti}
\end{theorem}

\begin{proof} Since $\Psi _x(f,t)\ge \tfrac 1m  \Psi 
_x(f^m,t)$, $x\in W(f^m)$ implies $x\in W(f)$ and 
similarly, $QW(f^m)\subseteq QW(f)$. Since (iii) follows by (i) and (ii), it suffices to prove that for every $\varepsilon >0$ there is a $\delta >0$ such that, for every prime integer $m$,
\begin{equation}
\label{eq11}
\Psi _x(f^m,\varepsilon) \ge\Psi _x(f,\delta ) \ \text{and} \ \Psi ^*_x(f^m,\varepsilon )\ge \Psi ^*_x(f,\delta ).
\end{equation}
For every $i\ge 0$, denote $\omega _i:=\omega (f^m,f^i(x))$
and $\omega _{ij}:=\omega _i\cap\omega _j$. 
Obviously, $\omega (f,x)=\bigcup _{0\le i<m}\omega _i$, and $f(\omega _{i})=\omega _{i+1}$, where $i$ is taken mod $m$. 
Moreover,  $f^m(\omega _i)=\omega _i$ and $f^m(\omega _{ij})=\omega _{ij}$, for every $0\le i<j<m$. Hence
\begin{equation}
\label{equ12}
\omega _i\ne \omega _{ij} \ \text{implies} \ \omega _j\ne\omega _{ij}, \ \text{and} \ f^i(x), f^j(x)\notin \omega _{ij}. 
\end{equation}
Let $k$ be the least period of $\omega _0$. Since $m$ is prime, there are two cases.

(a) If $k=m$ then the sets $\omega _i$ are pairwise distinct 
and, by (\ref{equ12}), there is a $\delta>0$ such that $B(x,\delta )\cap \omega _i=\emptyset$, $0<i<m$. It follows
that if $f^r(x)\in B(x,\delta )$ then $r$ is a multiple of $m$, with finitely many exceptions. Consequently, (\ref{eq11}) is satisfied for $\varepsilon =\delta$, even with $\ge$ replaced by the equality.

(b)  If $k=1$ then $\omega _i=\omega _0$, for every $i$.
Let $\varepsilon >0$. For every $i$, $0\le i<m$, there
is  the minimal integer $k_i\ge 0$ such that $f^{mk_i+i}(x)\in B(x,\varepsilon)$. By the continuity, there is a $\delta >0$
such that $f^{mk_i+i}(B(x,\delta ))\subseteq B(x,\varepsilon )$, $0\le i<m$. If $f^r(x)\in B(x,\delta )$ and $r\equiv i ({\rm mod}  \ m)$,  $r=ml+i$, then $f^{m(l+1+k_{m-i})}(x)=f^{r+ mk_{m-i}+{m-i}}(x)\in 
f^{mk_{m-i} +m-i}(B(x,\delta ))\subseteq B(x,\varepsilon )$. This proves (\ref{eq11}).
\end{proof}

In 2009  Z. Zhou and F. Li stated, among others,  the following problems, see \cite{zhou3}. 

\medskip

\noindent {\bf Problem 1.}   Does $IR(f)\ne\emptyset$ imply $h(f)>0$? 

\medskip

\noindent {\bf Problem 2.}  Does $W(f)\ne AP(f)$  imply $h(f)>0$?

\medskip

\noindent In general, the answer to either problem is negative.  In \cite{lenka} we constructed a  skew-product map $F:Q\times I\to Q\times I$, $(x,y)\mapsto (\tau (x), g_x(y))$, where $Q=\{ 0,1\}^{\mathbb N}$ is a Cantor-type set, $\tau$ the adding machine (or, odometer) on $Q$ and, for every $x$, $g_x$ is a nondecreasing mapping $I\to I$, with $g_x(0)=0$. Consequently, $h(F)=0$ and $Q_0:=Q\times\{ 0\}$ is an invariant set. On the other hand,  $IR(F)\ne\emptyset$ and $Q_0=AP(F)\ne W(F)$. This example answers in the negative both problems. 

However, for maps $f\in\mathcal C(I)$, $h(f)>0$ is equivalent to $IR(f)\ne\emptyset$. On the other hand, the answer to Problem 2 remains negative even for maps in $\mathcal C(I)$. Instead, we are able to show that  such maps with $W(f)\ne AP(f)$ are Li-Yorke chaotic. These results are given in the next section, as Theorems 2 and 3. Then, in Section 3 we show that these results can be extended to maps of more general one-dimensional compact metric space like topological graphs, topological trees, but not dendrites, see Theorems \ref{gen1} and \ref{gen2}. 

\section{Relations with topological entropy for maps in $\mathcal C(I)$}

\begin{theorem}
For $f\in\mathcal C (I)$, the conditions $h(f)>0$ and $IR(f)\ne\emptyset$ are equivalent.
\label{entropie}
\end{theorem}

\begin{proof} If $h(f)=0$ then $UR(f)=R(f)$ (see, e.g., \cite{block}, Corollary VI.8). Hence, by (\ref{eq10}), $W(f)=QW(f)$. If $h(f)>0$ then $W(f)\ne QW(f)$; this follows by Theorem \ref{dalsi3vlastnosti} and  Lemmas  \ref{shift} and \ref{kladna} stated below.
\end{proof}

Let $(\Sigma _2,\sigma )$ be the shift on the set $\Sigma _2$ of sequences of two symbols, $0, 1$, equipped with a metric $\rho$ of pointwise convergence, say, $\rho (\{x_i\}_{i\ge 1},\{y_i\}_{i\ge 1})=1/k$ where $k=\min \{ i\ge 1; x_i\ne y_i\}$.

\begin{lm} 
$IR(\sigma)$ is non-empty, and contains a transitive point.
\label{shift}
\end{lm}

\begin{proof}
Let
$$
k_{1,0},k_{1,1},k_{2,0},k_{2,1},k_{2,2},k_{3,0},\cdots 
,k_{3,3},k_{4,0},\cdots ,k_{4,4},k_{5,0},\cdots
$$
be an increasing sequence of positive integers.
Let $\{ B_n\} _{n\ge 1}$ be a sequence of all finite blocks of digits 0 and 1.
Put $A_0 = 10$, $A_1=(A_0)^{k_{1,0}}0^{k_{1,1}}B_1$ and, in general,  
\begin{equation}
\label{equ2}
A_{n}=A_{n-1}(A_0)^{k_{n,0}}(A_1)^{k_{n,1}}\cdots (A_{n-1})^{k_{n,n-1}}0^{k_{n,n}}B_{n}, \ n\ge 1.
\end{equation}
Denote by $|A|$ the lenght of a finite block of 0's and 1's, and let 
\begin{equation}
\label{equ20}
a_n=|A_n|, \ b_n=|B_n|,  \ c_n=a_n-b_{n}
-k_{n, n}, \ n\ge 1,
\end{equation}
and
\begin{equation}
\label{equ21}
 \lambda _{n,m}=\left|A_{n-1}(A_0)^{k_{n,0}}(A_1)^{k_{n,1}}\cdots (A_m)^{k_{n,m}}\right|, \  0\le m< n.
\end{equation}
By induction we can take the numbers $k_{i,j}$ such that
\begin{equation}
\label{equ22}
k_{n,m+1}=n\cdot \lambda _{n,m}, \  0\le m< n.
\end{equation}
Let $N(A)$ be the cylinder of all $x \in \Sigma_2$ beginning with a finite block $A$. Then $\{ N(B_n)\}_{n\ge 1}$ is a base of the topology of $\Sigma _2$, and $\bigcap _{n=1}^\infty N(A_n) $ contains exactly one point; denote it by $u$. 

Since  $\sigma ^{a_n-b_n}(u)\in N(B_n)$, i.e., since the trajectory of $u$ visits every $N(B_n)$, $u$ is a transitive point of $\sigma$. Moreover, $\rho (u, \sigma ^{j}
(u))=1$, whenever $c_n\le j<a_n-b_n$. By (\ref{equ22}) it follows that $\Psi _u(\sigma ,t)=0$ for every $t\in (0,1)$.
Consequently, $u\notin W(\sigma )$.

It remains to show that $u\in QW(\sigma )$.  Let $t\in (0,1)$. Fix an $n_0\in\mathbb N$ such that $1/a_{n_0}<t$.  
Then, by (\ref{equ2}),
$$
\#\left\{ j<\lambda _{n,n_0}; \rho (u,\sigma ^j(u))<t\right\} \ge k_{n,n_0}, \ n>n_0,
$$
hence, by (\ref{equ21}) and (\ref{equ22}),
$$
\lim _{n\to\infty}\frac{\#\left\{ j<\lambda _{n,n_0}; \rho (u,\sigma ^j(u))<t\right\}} {\lambda _{n,n_0}}\ge \lim _{n\to\infty}\frac{k_{n,n_0}}{\lambda _{n,n_0}}
=\lim_{n\to\infty}\frac{k_{n,n_0}}{\lambda _{n,n_0-1}+a_{n_0}k_{n,n_0}} =\lim _{n\to\infty}\frac n{1+a_{n_0}n}=\frac 1{a_{n_0}}.
$$
\smallskip
Thus, $\Psi^*_u(\sigma ,t)\ge 1/{a_{n_0}}$ and, by Lemma \ref{l1}, 
 $u\in QW(\sigma )$.
\end{proof}

\begin{lm}
Let $f\in \mathcal C (I)$ have positive topological entropy. Then $IR(f)\ne\emptyset$.
\label{kladna}
\end{lm}

\begin{proof}
When $h(f)>0$, then $f^m$ is strictly turbulent for some $m$. This means that there exist disjoint compact intervals $K_0$, $K_1$ such that $f^m(K_0)\cap f^m(K_1)\supset K_0\cup K_1$, see \cite{block}, Theorem IX.28. This condition is equivalent to the existence of a continuous map $g: X\subset I \rightarrow \Sigma_2$, where $X$ is of Cantor type, such that $g \circ f^m(x) = \sigma \circ g(x)$ for every $x \in X$, and such that each point in $\Sigma_2$ is the image of at most two points in $X$ (\cite{block}, Proposition II.15). By Lemma~\ref{shift}, there is a $u\in IR(\sigma )$. Hence, for every $t>0$, $\Psi_u^*(\sigma, t)>0$, and there is an $s>0$ such that $\Psi_u(\sigma, s)=0$. There are at most two preimages, $u_0$ and $u_1$, of $u$. Then, by the continuity,  $\Psi_{u_i}(f^m, r)=0$, for some $r>0$ and $i= 0, 1$, and $\Psi_{u_i}^*(f^m, k)>0$ for at least one $i\in\{0, 1\}$ and every $k>0$. Thus, $u_0\in IR(f^m)$ or $u_1\in IR(f^m )$ and, by Theorem \ref{dalsi3vlastnosti}, $IR (f)\ne\emptyset$.  
\end{proof}

Recall that $f\in\mathcal C (X)$ is {\it Li-Yorke chaotic}, or {\it LYC},   if there is an uncountable set $S\subseteq X$ such that, for every  $x\ne y$ in $S$, $\liminf _{n\rightarrow\infty} \rho (\varphi ^n(x), \varphi ^n(y))=0$ and  $\limsup _{n\rightarrow\infty} \rho (\varphi ^n(x), \varphi ^n(y))>0$.  

\begin{theorem}
\label{LiYorke}
For  $f\in \mathcal C(I)$, $W(f)\ne AP(f)$ implies that $f$ is Li-Yorke chaotic, but does not imply $h(f)>0$.
\end{theorem}

\begin{proof}
Every continuous map of a compact metric space with positive topological entropy is Li-Yorke chaotic \cite{BGKM}. Hence to prove the theorem  it suffices to consider  the class $\mathcal C_0\subset \mathcal C(I)$ of maps with zero topological entropy and show that 

(i) for every $f\in\mathcal C_0$, $W(f)\ne AP(f)$ implies  {\it LYC}, and

(ii) there is an $f\in\mathcal C_0$ with $W(f)\ne AP(f)$.

\noindent  For $f\in\mathcal C_0$,  $R(f)=UR(f)$, see, e.g., \cite{block}, Corollary VI.8. Hence,  by (\ref{eq10}), $W(f)\ne AP(f)$ implies that $f$ has an infinite minimal $\omega$-limit set $\widetilde\omega$ possessing a point which is not in $AP(f)$. 
Recall that for every such $\widetilde\omega$ there is an {\it associated system} $\{ J_n\}_{n\ge 1}$ of compact periodic intervals such that $J_n$ has period $2^n$, and $\widetilde\omega\subseteq \bigcap _{n\ge 1}\bigcup _{0\le j<2^n} f^{j}(J_n)$ \cite{smital}. For every $x\in\widetilde\omega$ there is a sequence $\iota (x)=\{j_n\}_{n\ge 1}$ of integers, $0\le j_n<2^n$, such that $x\in\bigcap _{n\ge 1}f^{j_n}(J_n)=:Q_x$. For every $x\in \widetilde\omega$, the set $\widetilde\omega \cap Q_x$ contains one (i.e., the point $x$) or two points. In the second case $Q_x=[a,b]$ is a compact wandering interval (i.e., $f^n(Q_x)\cap Q_x=\emptyset$ for every $n\ge 1$)  such that $a,b\in\widetilde\omega$ and either $x=a$ or $x=b$. Moreover, if, for every $x\in\widetilde\omega$,  $\widetilde\omega \cap Q_x$ is a singleton then $f$ restricted to $\widetilde\omega$ is the adding machine, and  $\widetilde\omega \subseteq AP(f)$, see \cite{brucknersmital}. Consequently, $W(f)\ne AP(f)$ implies the existence of an infinite $\omega$-limit set $\widetilde\omega$ such that
\begin{equation}
\label{eq12}
\widetilde\omega\cap Q_x=\{a,b\}, \ a<b, \ \text{for some} \ x\in\widetilde\omega . 
\end{equation} 
This condition characterizes {\it LYC} maps  in $\mathcal C_0$ (see \cite{smital} or subsequent books like \cite{block}) which proves (i).

To prove (ii) note that there are maps $f\in\mathcal C_0$ such that both $a$ and $b$ in (\ref{eq12}) are non-isolated points of $\widetilde\omega$, see \cite{brucknersmital} or \cite{MS}. Then $a,b\in UR(f)$ are minimal points. We show that in this case either $a\notin AP(f)$ or $b\notin AP(f)$ (actually, neither $a$ nor $b$ is in $AP(f)$ but we do not need this stronger property). So assume that $a,b\in AP(f)$ and  $U_a$, $U_b$ are their disjoint open neighborhoods. Then there is an {\it even} $m$,  $m=(2k+1)2^n$, with $n\ge 1$, such that $f^{jm}(a)\in U_a$ and $f^{jm}(b)\in U_b$, for every $j\ge 0$. Let $\{ J_n\} _{n\ge 1}$ be the  system of compact periodic intervals associated with $\widetilde\omega$. Without loss of generality we may assume that, for some $n$, $[a,b]\subset J_n$. Since $J_n$ has period $2^n$, for arbitrary odd $j$, $f^{jm}(J_n)\cap J_n=\emptyset$. If $f^{jm}(J_n)$ is to the left of $J_n$, then $f^{jm}(J_n)\cap U_b=\emptyset$, otherwise $f^{jm}(J_n)\cap U_a=\emptyset$. In any case, $f^{jm}(a)\notin U_a$ or $f^{jm}(b)\notin U_b$, which is a contradiction.
\end{proof}

\section{Generalization for maps on more general one-dimensional spaces}

Here we show that results given in Theorems \ref{entropie} and \ref{LiYorke} concerning maps in $\mathcal C(I)$  can be generalized to more general one-dimensional compact metric spaces like topological graphs or trees, but not dendrites. Recall that  $X$ is a {\it topological graph} if $X$ is a non-empty compact connected metric space which is the union of finitely many arcs (i.e., continuous images of the interval $I$) such that every two arcs can have only end-points in common. A {\it tree} is a topological graph which contains no subset homeomorphic to the circle. A {\it dendrite} is a locally connected continuum containing no subset homeomorphic to the circle. The proofs of generalized results are based on the same ideas, as the proofs of Theorems \ref{entropie} and \ref{LiYorke}. We only need some recent, nontrivial results concerning the structure of $\omega$-limit sets of such maps, see \cite{HrMa} and \cite{KKM}. Therefore we give here only outline of the proofs, pointing out only main differences.
\begin{theorem}
\label{gen1}
Let $f\in\mathcal C(X)$. 

{\rm (i)} \  If $X$ is a topological graph then $h(f)>0$ is equivalent to $QW(f)\ne W(f)$.

{\rm (ii)} \ There is a dendrit $X$ such that $h(f)>0$ and $QW(f)=W(f)=UR(f)$.
\end{theorem}

\begin{proof}
To prove (i) note that, for $f\in\mathcal C(X)$ where $X$ is a topological graph, $h(f)>0$ if and only if, for some $n\ge 1$, $f^n$ is turbulent \cite{HrMa}. Hence the proof of Lemma \ref{kladna} applies also to this case and  $h(f)>0$ implies $IR(f)\ne\emptyset$. On the other hand, if $h(f)=0$ then every infinite $\omega$-limit set is a solenoid (i.e., it has an associated system of compact periodic intervals $\{ J_n\}_{n\ge 1}$, $J_n$ with period $2^n$) and consequently, $R(f)=UR(f)$ \cite{HrMa} which gives the other implication.

(ii) In \cite{KKM} there is an example of a dendrit $X$ with a continuous map $f$ possessing exactly two $\omega$-limit sets: a minimal Cantor-type set $Q$ such that $h(f|_Q)\ge 0$
and a fixed point $p$ such that $\omega (f,x)=\{ p\}$ for every $x\in X\setminus Q$.
\end{proof}

\begin{theorem}
\label{gen2}
Let $f\in\mathcal C(X)$.

{\rm (i)} \ If $X$ is a compact tree then $W(f)\ne AP(f)$ implies LYC, but does not imply $h(f)>0$.

{\rm (ii)} \ If $X$ is a dendrit, or a topological graph containing a circle then $W(f)\ne AP(f)$ implies neither LYC nor $h(f)>0$.
\end{theorem}

\begin{proof}
(i) Similarly as in the proof of Theorem \ref{LiYorke} we may assume that $h(f)=0$. Then every infinite $\omega$-limit set of $f$ is a solenoid and the argument, with obvious modifications, applies. 

(ii) If $X$ is the circle, take $f$ to be an irrational rotation. Then obvioulsy $X=UR(f)\setminus AP(f)=W(f)\setminus AP(f)$ but $f$ is not {\it LYC}. On the other hand, 
let $\widetilde\omega$ be the $\omega$-limit set used in the proof of part (ii) of Theorem \ref{LiYorke}. Thus, $\widetilde\omega$ is a minimal set intersecting $UR(f)\setminus AP(f)$. A modification of the construction from \cite{KKM} yields a dendrite with exactly two $\omega$-limit sets, an infinite minimal set $Q=\widetilde\omega$ and a fixed point $q$ (see proof of part (ii) of preceding theorem).  It is easy to see that $f$ is not {\it LYC}.
\end{proof}

\begin{remark}
By Theorems \ref{gen1} and \ref{gen2}, for a map $f\in\mathcal C(X)$ where $X$ is a compact metric space, the properties $h(f)>0$ and $W(f)\ne AP(f)$ are independent.
Similarly, $h(f)>0$ and $IR(f)\ne\emptyset$ are independent. Example of a map $f$ with $h(f)=0$ and $IR(f)\ne\emptyset$ is given in \cite{lenka} (see also text at the end of Section 1), and any  minimal map $f$ with $h(f)>0$ yields $IR(f)=\emptyset$.
\end{remark}

\bigskip

{\bf Acknowledgments}

\bigskip

The author thanks Professor Jaroslav Sm\'{\i}tal for his heedful guidance and helpful suggestions.

 \end{document}